\documentclass[12pt]{amsart}

\usepackage{latexsym,amsmath,amssymb,amsthm}
\usepackage{graphicx,comment}
\usepackage[mathscr]{eucal}

\newcommand{\wh}{\widehat}                  
\newcommand{\wt}{\widetilde}                

\newtheorem{thm}{Theorem}[section]
\newtheorem{lem}[thm]{Lemma}
\newtheorem{prop}[thm]{Proposition}
\newtheorem{cor}[thm]{Corollary}

\theoremstyle{definition}
\newtheorem{defn}[thm]{Definition}
\newtheorem{exam}[thm]{Example}
\theoremstyle{remark}
\newtheorem{rem}[thm]{Remark}

\begin{document}

\title     {
                On the skein polynomial for links
           }
\author    {
                             Boju Jiang
           }
\author    {
                             Jiajun Wang
           }
\author    {
                             Hao Zheng
           }
\address   {
                      Department of Mathematics\\
                      Peking University\\
                      Beijing  100871\\
                      China
           }
\email     {
                      bjjiang@math.pku.edu.cn
           }
\email     {
                      wjiajun@math.pku.edu.cn
           }
\email     {
                      hzheng@math.pku.edu.cn
           }
\thanks    {Partially supported by NSFC grant \#11131008}



\date{}

\begin{abstract}
We give characterizations of the skein polynomial for links
(as well as Jones and Alexander-Conway polynomials derivable from it),
avoiding the usual ``smoothing of a crossing'' move.
As by-products we have characterizations of these polynomials for knots,
and for links with any given number of components.
\end{abstract}

\maketitle

\section {Introduction}
\label{sec:Intro}

The skein polynomial (as called in \cite[Chapter 8]{K1},
also known as HOMFLY or HOMFLY-PT polynomial),
$P_L(a,z) \in \mathbb Z[a^{\pm1},z^{\pm1}]$, is an invariant for oriented links.
Here $\mathbb Z[a^{\pm1},z^{\pm1}]$ is the ring of Laurent polynomials in two variables $a$ and $z$, with integer coefficients.
It is defined to be the invariant of oriented links satisfying the axioms
\begin{gather*}
a^{-1}\cdot P
\begin{pmatrix}
\includegraphics[width=.8cm]{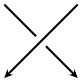}
\end{pmatrix}
-a\cdot P
\begin{pmatrix}
\includegraphics[width=.8cm]{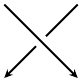}
\end{pmatrix}
=z\cdot P
\begin{pmatrix}
\includegraphics[width=.8cm]{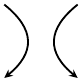}
\end{pmatrix} ;
\tag*{\rm(I)}
\\
P
\begin{pmatrix}
\includegraphics[width=.8cm]{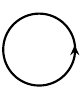}
\end{pmatrix}
=1 .
\tag*{\rm(O)}
\end{gather*}

The Alexander-Conway polynomial $\Delta_L \in \mathbb Z[t^{\pm\frac12}]$ and
the Jones polynomial $V_L \in \mathbb Z[t^{\pm\frac12}]$ are related to the skein polynomial:
\[
\Delta_L(t)=P_L(1,t^{\frac12}-t^{-\frac12}),
\qquad
V_L(t)=P_L(t,t^{\frac12}-t^{-\frac12}).
\]

Our main result is

\begin{thm}
\label{thm:HOMFLY}
The skein polynomial $P_L \in \mathbb Z[a^{\pm1},z^{\pm1}]$ is
the invariant of oriented links
determined uniquely by the following four axioms.
{\allowdisplaybreaks
\begin{gather*}
a^{-2}\cdot P
\begin{pmatrix}
\includegraphics[width=.8cm]{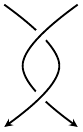}
\end{pmatrix}
+a^2\cdot P
\begin{pmatrix}
\includegraphics[width=.8cm]{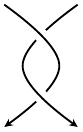}
\end{pmatrix}
=(2+z^2)
\cdot P
\begin{pmatrix}
\includegraphics[width=.8cm]{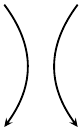}
\end{pmatrix} ;
\tag*{\rm(II)}
\\
a^{-1}\cdot P
\begin{pmatrix}
\includegraphics[width=1.2cm]{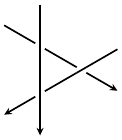}
\end{pmatrix}
-a\cdot P
\begin{pmatrix}
\includegraphics[width=1.2cm]{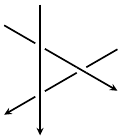}
\end{pmatrix}
=
a^{-1}\cdot P
\begin{pmatrix}
\includegraphics[width=1.2cm]{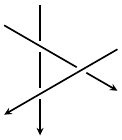}
\end{pmatrix}
-a\cdot P
\begin{pmatrix}
\includegraphics[width=1.2cm]{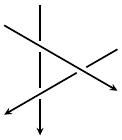}
\end{pmatrix} ;
\tag*{\rm(III)}
\\
P
\begin{pmatrix}
\includegraphics[height=1.2cm]{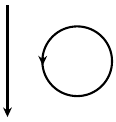}
\end{pmatrix}
=z^{-1}(a^{-1}-a)\cdot
P
\begin{pmatrix}
\;\;
\includegraphics[height=1.2cm]{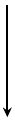}
\;\;
\end{pmatrix} ;
\tag*{\rm(IO)}
\\
P
\begin{pmatrix}
\includegraphics[width=.8cm]{fig_012/O-0}
\end{pmatrix}
=1 .
\tag*{\rm(O)}
\end{gather*}
}
\end{thm}

A parallel result is for the Jones polynomial.
It is not a direct corollary of the above theorem,
because the substitutions $a\mapsto t$ and $z\mapsto (t^{\frac12}-t^{-\frac12})$
do not send $\mathbb Z[a^{\pm1},z^{\pm1}]$ into $Z[t^{\pm\frac12}]$.

\begin{thm}
\label{thm:Jones}
The Jones polynomial $V_L \in \mathbb Z[t^{\pm\frac12}]$ is
the invariant of oriented links
determined uniquely by the following four axioms.
{\allowdisplaybreaks
\begin{gather*}
t^{-2}\cdot V
\begin{pmatrix}
\includegraphics[width=.8cm]{fig_012/1a1a_dd-0}
\end{pmatrix}
+t^2\cdot V
\begin{pmatrix}
\includegraphics[width=.8cm]{fig_012/1b1b_dd-0}
\end{pmatrix}
=(t+t^{-1})
\cdot V
\begin{pmatrix}
\includegraphics[width=.8cm]{fig_012/1oo_dd-0}
\end{pmatrix} ;
\tag*{\rm(II$_V$)}
\\
t^{-1}\cdot V
\begin{pmatrix}
\includegraphics[width=1.2cm]{fig_3icons/1a2a1b_ddd-0}
\end{pmatrix}
-t\cdot V
\begin{pmatrix}
\includegraphics[width=1.2cm]{fig_3icons/1a2b1b_ddd-0}
\end{pmatrix}
=
t^{-1}\cdot V
\begin{pmatrix}
\includegraphics[width=1.2cm]{fig_3icons/1b2a1a_ddd-0}
\end{pmatrix}
-t\cdot V
\begin{pmatrix}
\includegraphics[width=1.2cm]{fig_3icons/1b2b1a_ddd-0}
\end{pmatrix} ;
\tag*{\rm(III$_V$)}
\\
V
\begin{pmatrix}
\includegraphics[height=1.2cm]{fig_012/IO-0}
\end{pmatrix}
=-(t^{\frac12}+t^{-\frac12})\cdot
V
\begin{pmatrix}
\;\;
\includegraphics[height=1.2cm]{fig_012/I-0}
\;\;
\end{pmatrix} ;
\tag*{\rm(IO$_V$)}
\\
V
\begin{pmatrix}
\includegraphics[width=.8cm]{fig_012/O-0}
\end{pmatrix}
=1 .
\tag*{\rm(O$_V$)}
\end{gather*}
}
\end{thm}

For the Alexander-Conway polynomial, the result takes a slightly different form.
We switch to a ($\Phi$)-type axiom because the (IO)-type one degenerates into a
consequence of (II) and (III) (see Corollary~\ref{cor:stabilization_Alexander-Conway}).

\begin{thm}
\label{thm:Alexander-Conway}
The Alexander-Conway polynomial $\Delta_L \in \mathbb Z[t^{\pm\frac12}]$ is
the invariant of oriented links
determined uniquely by the following four axioms.
{\allowdisplaybreaks
\begin{gather*}
\Delta
\begin{pmatrix}
\includegraphics[width=.8cm]{fig_012/1a1a_dd-0}
\end{pmatrix}
+ \Delta
\begin{pmatrix}
\includegraphics[width=.8cm]{fig_012/1b1b_dd-0}
\end{pmatrix}
=(t+t^{-1})
\cdot \Delta
\begin{pmatrix}
\includegraphics[width=.8cm]{fig_012/1oo_dd-0}
\end{pmatrix} ;
\tag*{\rm(II$_\Delta$)}
\\
\Delta
\begin{pmatrix}
\includegraphics[width=1.2cm]{fig_3icons/1a2a1b_ddd-0}
\end{pmatrix}
- \Delta
\begin{pmatrix}
\includegraphics[width=1.2cm]{fig_3icons/1a2b1b_ddd-0}
\end{pmatrix}
=
\Delta
\begin{pmatrix}
\includegraphics[width=1.2cm]{fig_3icons/1b2a1a_ddd-0}
\end{pmatrix}
- \Delta
\begin{pmatrix}
\includegraphics[width=1.2cm]{fig_3icons/1b2b1a_ddd-0}
\end{pmatrix} ;
\tag*{\rm(III$_\Delta$)}
\\
\Delta
\begin{pmatrix}
\includegraphics[height=1.2cm]{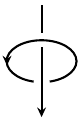}
\end{pmatrix}
=(t^{\frac12}-t^{-\frac12})\cdot
\Delta
\begin{pmatrix}
\;\;
\includegraphics[height=1.2cm]{fig_012/I-0}
\;\;
\end{pmatrix} ;
\tag*{($\Phi_\Delta$)}
\\
\Delta
\begin{pmatrix}
\includegraphics[width=.8cm]{fig_012/O-0}
\end{pmatrix}
=1 .
\tag*{\rm(O$_\Delta$)}
\end{gather*}
}
\end{thm}

If we restrict our attention to oriented links with a fixed number $\mu>0$ of components,
the axiom (IO) becomes irrelevant but we must pick a suitable normalization.
Let $U_\mu$ denote the $\mu$-component unlink, and
let $C_\mu$ denote the $\mu$-component oriented chain
where adjacent rings have linking number $+1$.
(In terms of closed braids, $U_\mu$ is the closure of the trivial braid $e\in B_\mu$,
and $C_\mu$ is the closure of the braid $\sigma_1^2\sigma_2^2\dots\sigma_{\mu-1}^2\in B_\mu$.)
We can use either $U_\mu$ or $C_\mu$ (but $U_\mu$ is preferred) to normalize the skein or Jones polynomial,
but for Alexander-Conway polynomial we can only use $C_\mu$.

\begin{thm}
\label{thm:HOMFLY_knot}
The skein polynomial $P_L$ is
the invariant of oriented $\mu$-compo\-nent links
determined uniquely by the following three axioms.
{\allowdisplaybreaks
\begin{gather*}
a^{-2}\cdot P
\begin{pmatrix}
\includegraphics[width=.8cm]{fig_012/1a1a_dd-0}
\end{pmatrix}
+a^2\cdot P
\begin{pmatrix}
\includegraphics[width=.8cm]{fig_012/1b1b_dd-0}
\end{pmatrix}
=(2+z^2)
\cdot P
\begin{pmatrix}
\includegraphics[width=.8cm]{fig_012/1oo_dd-0}
\end{pmatrix} ;
\tag*{\rm(II)}
\\
a^{-1}\cdot P
\begin{pmatrix}
\includegraphics[width=1.2cm]{fig_3icons/1a2a1b_ddd-0}
\end{pmatrix}
-a\cdot P
\begin{pmatrix}
\includegraphics[width=1.2cm]{fig_3icons/1a2b1b_ddd-0}
\end{pmatrix}
=
a^{-1}\cdot P
\begin{pmatrix}
\includegraphics[width=1.2cm]{fig_3icons/1b2a1a_ddd-0}
\end{pmatrix}
-a\cdot P
\begin{pmatrix}
\includegraphics[width=1.2cm]{fig_3icons/1b2b1a_ddd-0}
\end{pmatrix} ;
\tag*{\rm(III)}
\\[2ex]
P(U_\mu) = (z^{-1}(a^{-1}-a))^{\mu-1} .
\tag*{\rm(U)}
\end{gather*}
}
\end{thm}

\begin{thm}
\label{thm:Jones_knot}
The Jones polynomial $V_K$ is
the invariant of oriented $\mu$-component links
determined uniquely by the following three axioms.
{\allowdisplaybreaks
\begin{gather*}
t^{-2}\cdot V
\begin{pmatrix}
\includegraphics[width=.8cm]{fig_012/1a1a_dd-0}
\end{pmatrix}
+t^2\cdot V
\begin{pmatrix}
\includegraphics[width=.8cm]{fig_012/1b1b_dd-0}
\end{pmatrix}
=(t+t^{-1})
\cdot V
\begin{pmatrix}
\includegraphics[width=.8cm]{fig_012/1oo_dd-0}
\end{pmatrix} ;
\tag*{\rm(II$_V$)}
\\
t^{-1}\cdot V
\begin{pmatrix}
\includegraphics[width=1.2cm]{fig_3icons/1a2a1b_ddd-0}
\end{pmatrix}
-t\cdot V
\begin{pmatrix}
\includegraphics[width=1.2cm]{fig_3icons/1a2b1b_ddd-0}
\end{pmatrix}
=
t^{-1}\cdot V
\begin{pmatrix}
\includegraphics[width=1.2cm]{fig_3icons/1b2a1a_ddd-0}
\end{pmatrix}
-t\cdot V
\begin{pmatrix}
\includegraphics[width=1.2cm]{fig_3icons/1b2b1a_ddd-0}
\end{pmatrix} ;
\tag*{\rm(III$_V$)}
\\[2ex]
V(U_\mu) = (-(t^{\frac12}+t^{-\frac12}))^{\mu-1} .
\tag*{\rm(U$_V$)}
\end{gather*}
}
\end{thm}

\begin{thm}
\label{thm:Alexander-Conway_knot}
The Alexander-Conway polynomial $\Delta_K$ is
the invariant of oriented $\mu$-component links
determined uniquely by the following three axioms.
{\allowdisplaybreaks
\begin{gather*}
\Delta
\begin{pmatrix}
\includegraphics[width=.8cm]{fig_012/1a1a_dd-0}
\end{pmatrix}
+ \Delta
\begin{pmatrix}
\includegraphics[width=.8cm]{fig_012/1b1b_dd-0}
\end{pmatrix}
=(t+t^{-1})
\cdot \Delta
\begin{pmatrix}
\includegraphics[width=.8cm]{fig_012/1oo_dd-0}
\end{pmatrix} ;
\tag*{\rm(II$_\Delta$)}
\\
\Delta
\begin{pmatrix}
\includegraphics[width=1.2cm]{fig_3icons/1a2a1b_ddd-0}
\end{pmatrix}
- \Delta
\begin{pmatrix}
\includegraphics[width=1.2cm]{fig_3icons/1a2b1b_ddd-0}
\end{pmatrix}
=
\Delta
\begin{pmatrix}
\includegraphics[width=1.2cm]{fig_3icons/1b2a1a_ddd-0}
\end{pmatrix}
- \Delta
\begin{pmatrix}
\includegraphics[width=1.2cm]{fig_3icons/1b2b1a_ddd-0}
\end{pmatrix} ;
\tag*{\rm(III$_\Delta$)}
\\
\Delta(C_\mu) = (t^{\frac12}-t^{-\frac12})^{\mu-1} .
\tag*{\rm(C$_\Delta$)}
\end{gather*}
}
\end{thm}

Note that the foundational relation (I) cannot appear in
Theorems~\ref{thm:HOMFLY_knot}--\ref{thm:Alexander-Conway_knot}
because it involves links with different number of components.

Our approach is via closed braids.
We explain the language of relators in Section~\ref{sec:relators} and
give an algebraic reduction lemma in Section~\ref{sec:key_lem}.
This approach is adapted from the corresponding sections of \cite{J1}
on Conway's potential function for colored links.
The current context of uncolored links makes the reduction argument more transparent.
Section~\ref{sec:stabilizations} discusses closed braids with different number of strands.
The theorems are proved in the last two sections.

\section{Braids and skein relators}
\label{sec:relators}

For braids, we use the following conventions:
Braids are drawn from top to bottom.
The strands of a braid are numbered at the top of the braid, from left to right.
The product $\beta_1\cdot\beta_2$ of two $n$-braids is obtained by drawing $\beta_2$ below $\beta_1$.
The set $B_n$ of all $n$-braids forms a group under this multiplication,
with standard generators $\sigma_1,\sigma_2,\dots,\sigma_{n-1}$.

It is well known that links can be presented as closed braids.
The closure of a braid $\beta\in B_n$ will be denoted $\wh\beta$.
Two braids (possibly with different number of strands) have isotopic closures
if and only if they can be related by a finite sequence of two types of moves:
\begin{enumerate}
\item Conjugacy move: $\beta$ $\leftrightsquigarrow$ $\beta'$ where $\beta,\beta'$ are conjugate in a braid group $B_n$;
\item Markov move: $\beta\in B_n$ $\leftrightsquigarrow$ $\beta\sigma_n^{\pm1}\in B_{n+1}$.
\end{enumerate}

Let $\Lambda$ be the Laurent polynomial ring $\mathbb Z[a^{\pm1},z^{\pm1}]$.
Let $\Lambda B_n$ be the group-algebra on $B_n$ with coefficients in $\Lambda$.

\begin{defn}
\label{defn:skein_relation}
We say that an element
\[
\lambda_1\cdot\beta_1+\dots+\lambda_k\cdot\beta_k 
\]
of $\Lambda B_n$ is a \emph{skein relator}, or equivalently, say that the corresponding formal equation
(in which $P_{L_{\beta_h}}$ stands for the $P$ of the link $L_{\beta_h}$)
\[
\lambda_1\cdot P_{L_{\beta_1}}+\dots+\lambda_k\cdot P_{L_{\beta_k}}=0
\]
is a \emph{skein relation}, if the following condition is satisfied:
For any links $L_{\beta_1},\dots,L_{\beta_k}$ that are identical except in a cylinder
where they are represented by the braids $\beta_1,\dots,\beta_k$ respectively,
the formal equation becomes an equality in $\Lambda$.
\end{defn}

\begin{exam}
\label{exam:relator_vs_relation}
To every element $\lambda_1\cdot\beta_1+\dots+\lambda_k\cdot\beta_k\in \Lambda B_n$,
by taking braid closures we have a corresponding element
\[
\lambda_1\cdot P_{\wh\beta_1}+\dots +\lambda_k\cdot P_{\wh\beta_k} \in \Lambda.
\]
The latter vanishes if the former is a skein relator.
\end{exam}

\begin{exam}
The skein relations (I), (II) and (III) in Section~\ref{sec:Intro}
correspond to the following relators, respectively:
(The symbol $e$ stands for the trivial braid.)
{\allowdisplaybreaks
\begin{gather*}
\text{\rm(I$_\text{B}$)}:=
a^{-1}\cdot\sigma_1^2 -a\cdot\sigma_1^{-2} -z\cdot{e} ;
\\
\text{\rm(II$_\text{B}$)}:=
a^{-2}\cdot\sigma_1^2 +a^2\cdot\sigma_1^{-2} -(2+z^2)\cdot{e} ;
\\
\text{\rm(III$_\text{B}$)}:=
\begin{aligned}[t]
&a^{-1}\cdot{\sigma_1\sigma_2\sigma_1^{-1}}
+ a\cdot{\sigma_1^{-1}\sigma_2^{-1}\sigma_1}
\\
&- a^{-1}\cdot{\sigma_1^{-1}\sigma_2\sigma_1}
- a\cdot{\sigma_1\sigma_2^{-1}\sigma_1^{-1}} .
\end{aligned}
\end{gather*}
}
\end{exam}


\begin{prop}
\label{prop:relator_ideal}
Assume that
\[
\lambda_1\cdot P_{L_{\beta_1}}+\dots+\lambda_k\cdot P_{L_{\beta_k}}=0
\]
is a skein relation. Then for any given braid $\alpha\in B_n$,
the following equations are also skein relations:
\begin{gather*}
\lambda_1\cdot P_{L_{(\beta_1\alpha)}}+\dots+\lambda_k\cdot P_{L_{(\beta_k\alpha)}}=0;
\\
\lambda_1\cdot P_{L_{(\alpha\beta_1)}}+\dots+\lambda_k\cdot P_{L_{(\alpha\beta_k)}}=0.
\end{gather*}

Hence skein relators form a two-sided ideal\/ $\mathfrak R_n$
(called the \emph{relator ideal\/})
in $\Lambda B_n$.
\end{prop}

\begin{proof}
Look at the cylinder where the links $L_{(\beta_1\alpha)},\dots,L_{(\beta_k\alpha)}$
are represented differently by braids $\beta_1\alpha,\dots,\beta_k\alpha$, respectively.
In the upper half cylinder they are represented by braids $\beta_1,\dots,\beta_k$.
So the assumption implies the first equality.
Similarly for the second equality.

By Definition~\ref{defn:skein_relation}, this means skein relators form a two-sided ideal.
\end{proof}

\section{An algebraic reduction lemma}
\label{sec:key_lem}

\begin{defn}
\label{defn:equivalence}
Let $\mathfrak I_n$ be the two-sided ideal in $\Lambda B_n$
generated by $\text{\rm(II$_\text{B}$)}$ and $\text{\rm(III$_{\text{B}}$)}$.
(When $n=2$ we ignore $\text{\rm(III$_{\text{B}}$)}$.)

Two elements of the algebra $\Lambda B_n$ are \emph {equivalent modulo $\mathfrak I_n$}
(denoted by $\sim$\,) if
their difference is in $\mathfrak I_n$.
\end{defn}

For example, by conjugation in $B_n$ we have
$a^{-2}\cdot \sigma_i^2 +a^2\cdot \sigma_i^{-2} -(2+z^2)\cdot e \sim 0$ and
$a^{-1}\cdot\sigma_i\sigma_{i+1}\sigma_i^{-1}
+a\cdot \sigma_i^{-1}\sigma_{i+1}^{-1}\sigma_i
-a^{-1}\cdot \sigma_i^{-1}\sigma_{i+1}\sigma_i
-a\cdot \sigma_i\sigma_{i+1}^{-1}\sigma_i^{-1}
\sim 0$,
for any $i$.

\begin{lem}
\label{lem:reduction}
Modulo $\mathfrak I_n$, every braid $\beta\in B_n$ is
equivalent to a $\Lambda$-linear combination of braids of the form
$\alpha\sigma_{n-1}^k \gamma$ with $\alpha,\gamma\in B_{n-1}$ and $k\in\{0,\pm1,2\}$.
\end{lem}

A braid $\beta\in B_n$ can be written as
\[
\beta=\beta_0\sigma_{n-1}^{k_1} \beta_1\sigma_{n-1}^{k_2} \dots \sigma_{n-1}^{k_r}\beta_r
\]
where $\beta_j\in B_{n-1}$ and $k_j\neq0$.
We allow that $\beta_0$ and $\beta_r$ be trivial,
but assume other $\beta_j$'s are nontrivial.
The number $r$ will be denoted as $r(\beta)$.

The lemma will be proved by an induction on the double index $(n,r)$.
Note that the lemma is trivial when $n=2$, or $r(\beta)\leq1$.

It is enough to consider the case $r=2$, because induction on $r$ works beyond $2$.
Indeed, if $r(\beta)>2$, let $\beta'=\beta_1\sigma_{n-1}^{k_2} \dots \sigma_{n-1}^{k_r}\beta_r$, then
$r(\beta')<r(\beta)$. By inductive hypothesis $\beta'$ is equivalent to a
linear combination of elements of the form $\alpha'\sigma_{n-1}^{k'} \gamma'$,
hence $\beta$ is equivalent to a linear combination of elements of the form
$\beta_0\sigma_{n-1}^{k_1} \alpha'\sigma_{n-1}^{k'} \gamma'$.
This brings the problem back to the $r=2$ case.
Henceforth we assume $r=2$.

Since the initial and terminal part of $\beta$, namely $\beta_0$ and $\beta_r$,
do not affect the conclusion of the lemma, we can drop them.
So we assume $\beta=\sigma_{n-1}^{k_1} \beta_1\sigma_{n-1}^{k_2}$, where $\beta_1\in B_{n-1}$.

By the induction hypothesis on $n$, $\beta_1\in B_{n-1}$ is a linear combination of
elements of the form $\alpha_1\sigma_{n-2}^\ell \gamma_1$.
Note that $\alpha_1,\gamma_1\in B_{n-2}$ commute with $\sigma_{n-1}$.
So it suffices to focus on braids of the form
$\beta=\sigma_{n-1}^{k}\sigma_{n-2}^{\ell}\sigma_{n-1}^{m}$.

For the sole purpose of controlling the length of displayed formulas,
we assume $n=3$ below.
The proof for a general $n$ can be obtained by a simple change of subscripts,
replacing $\sigma_1,\sigma_2$ with $\sigma_{n-2},\sigma_{n-1}$ and
replacing $t_1,t_2,t_3$ with $t_{n-2},t_{n-1},t_{n}$, respectively.

Thus, Lemma~\ref{lem:reduction} has been reduced to the following
\begin{lem}
\label{lem:key_reduction}
Every $\sigma_2^{k}\sigma_1^{\ell}\sigma_2^{m}$ is equivalent \textup{(modulo $\mathfrak I_n$)} to a linear combination
of braids of the form $\sigma_1^{k'}\sigma_2^{\ell'}\sigma_1^{m'}$ where $\ell'$ is $0$, $\pm1$ or $2$.
\end{lem}

\begin{proof}
Modulo $\text{(II$_\text{B}$)}$, we may restrict the exponent
$k$ to take values $1$, $2$ and $3$ (we are done if $k$ is $0$).
If $k>1$ we can decrease $k$ by looking at
$\sigma_2^{k-1}(\sigma_2\sigma_1^{\ell}\sigma_2^m)$, so it suffices to prove the case $k=1$.
Again modulo $\text{(II$_\text{B}$)}$, we can restrict the exponents $\ell, m$ to the values $\pm1$ and $2$.
There are altogether 9 cases to verify.

{\it 5 trivial cases (braid identities) }:
\begin{alignat*}{3}
&\sigma_2\sigma_1\sigma_2=\sigma_1\sigma_2\sigma_1, \quad
&&\sigma_2\sigma_1\sigma_2^{-1}=\sigma_1^{-1}\sigma_2\sigma_1, \quad
&&\sigma_2\sigma_1^{-1}\sigma_2^{-1}=\sigma_1^{-1}\sigma_2^{-1}\sigma_1, \quad
\\
&\sigma_2\sigma_1\sigma_2^2=\sigma_1^2\sigma_2\sigma_1, \quad
&&\sigma_2\sigma_1^2\sigma_2^{-1}=\sigma_1^{-1}\sigma_2^2\sigma_1. \quad
\end{alignat*}

{\it The case $\sigma_2\sigma_1^{-1}\sigma_2$ }:
Multiplying $\text{\rm(III$_{\text{B}}$)}$ by $\sigma_2$ on the right and $\sigma_1^{-1}$ on the left,
and taking braid identities into account, we get the relation
\[
a^{-1}\cdot \sigma_2\sigma_1^{-1}\sigma_2
+ a\cdot \sigma_1^{-1}\sigma_2\sigma_1^{-1}
- a^{-1}\cdot \sigma_1^{-1}\sigma_2\sigma_1
- a\cdot \sigma_1\sigma_2^{-1}\sigma_1^{-1}
\sim 0.
\]
Then $\sigma_2\sigma_1^{-1}\sigma_2$ is equivalent to a linear combination of
braids of the form $\sigma_1^{\pm1}\sigma_2^{\pm1}\sigma_1^{\pm1}$.
So the case $\sigma_2\sigma_1^{-1}\sigma_2$ is verified.

{\it The case $\sigma_2\sigma_1^{-1}\sigma_2^2$ }:
Multiplying the previous relation by $\sigma_2$ on the right,
and taking braid identities into account, we see that
\[
a^{-1}\cdot \sigma_2\sigma_1^{-1}\sigma_2^2
+ a\cdot \sigma_1^{-1}(\sigma_2\sigma_1^{-1}\sigma_2)
- a^{-1}\cdot \sigma_2\sigma_1
- a\cdot \sigma_1^2\sigma_2^{-1}\sigma_1^{-1}
\sim 0.
\]
Similar to the above case, this reduces $\sigma_2\sigma_1^{-1}\sigma_2^2$
to the verified case $\sigma_2\sigma_1^{-1}\sigma_2$.

{\it The case $\sigma_2\sigma_1^2\sigma_2$ }:
Multiplying $\text{\rm(III$_{\text{B}}$)}$ on the right by $\sigma_1\sigma_2\sigma_1$, we get
\[
a^{-1}\cdot \sigma_1\sigma_2^2\sigma_1
+ a\cdot \sigma_2^2
- a^{-1}\cdot \sigma_2\sigma_1^2\sigma_2
- a\cdot \sigma_1^2
\sim 0.
\tag*{$\text{\rm(III$'_{\text{B}}$)}$}
\]
This verifies the case $\sigma_2\sigma_1^2\sigma_2$.

{\it The case $\sigma_2\sigma_1^2\sigma_2^2$ }:
Multiplying $\text{\rm(III$'_{\text{B}}$)}$ by $\sigma_2$ on the right,
we get
\[
a^{-1}\cdot \sigma_1^2\sigma_2\sigma_1^2
+ a\cdot \sigma_2^3
- a^{-1}\cdot \sigma_2\sigma_1^2\sigma_2^2
- a\cdot \sigma_1^2\sigma_2
\sim 0.
\]
The case $\sigma_2\sigma_1^2\sigma_2^2$ is also verified.

We have verified all 9 cases.
Modulo $\text{\rm(II$_\text{B}$)}$ we can assume $\ell'\in\{0,\pm1,2\}$.
Thus Lemma~\ref{lem:key_reduction} is proved.

The inductive proof of Lemma~\ref{lem:reduction} is now complete.
\end{proof}

The resulting $\Lambda$-linear combination of braids of the form
$\alpha\sigma_{n-1}^k \gamma$ with $\alpha,\gamma\in B_{n-1}$ in the Lemma is not unique,
but the inductive proof gives us a recursive algorithm to find one.

To compare the ideal $\mathfrak I_n$ with the relator ideal $\mathfrak R_n$ of Section~\ref{sec:relators}, we have

\begin{prop}
$\mathfrak I_n\subset \mathfrak R_n$ but
$\mathfrak I_n\neq \mathfrak R_n$.
\end{prop}

\begin{proof}
The inclusion is easy.
Indeed, $(\text{I}_\text{B})$ is in the relator ideal $\mathfrak R_n$, and
\begin{gather*}
\text{\rm(II$_\text{B}$)}=
\text{\rm(I$_\text{B}$)}^2 +2z\cdot \text{\rm(I$_\text{B}$)} ,
\\
\text{\rm(III$_\text{B}$)}=
\sigma_2^{-1}\cdot \text{\rm(I$_\text{B}$)} \cdot\sigma_2 -\sigma_2\cdot\text{\rm(I$_\text{B}$)} \cdot\sigma_2^{-1} .
\end{gather*}
So both $\text{\rm(II$_\text{B}$)}$ and $\text{\rm(III$_\text{B}$)}$ are in $\mathfrak R_n$.
Therefore $\mathfrak I_n\subset \mathfrak R_n$.

To show they are not equal, we need the notion of homogeneity.
Each $n$-braid $\beta$ has an \emph{underlying permutation} of $\{1,\dots,n\}$, denoted $i\mapsto i^{\beta}$,
where $i^{\beta}$ is the position of the $i$-th strand at the bottom of $\beta$.
In this way the braid group $B_n$ projects onto the symmetric group $\mathfrak S_n$.
An element of $\Lambda B_n$ is called \emph{homogeneous} if all its terms
(with nonzero coefficients) have the same underlying permutation.
As a $\Lambda$-module, $\Lambda B_n$ splits into a direct sum
according to underlying permutations of braids.
Under this splitting, every element of $\Lambda B_n$ decomposes into a sum of its
\emph{homogeneous components}.

Since $\text{\rm(II$_\text{B}$)}$ and $\text{\rm(III$_\text{B}$)}$ are homogeneous,
the ideal $\mathfrak I_n\subset\Lambda B_n$ is generated by homogeneous elements.
Then every homogeneous component of any element of $\mathfrak I_n$ is also in $\mathfrak I_n$.
Now the relator $\text{\rm(I$_\text{B}$)}\in \mathfrak R_n$ has a homogeneous component $-z\cdot e$
which is not a relator.
Hence $\text{\rm(I$_\text{B}$)}$ is not in $\mathfrak I_n$.
Thus $\mathfrak I_n$ is strictly smaller than $\mathfrak R_n$.
\end{proof}

\section{Stabilizations}
\label{sec:stabilizations}

Suppose a braid $\beta\in B_n$ is written as a word in the standard generators
$\sigma_1$, $\sigma_2$, \dots, $\sigma_{n-1}$.
The same word $\beta$ gives a braid in $B_{n+k}$ for any $k\ge0$.
Thus $B_n$ is standardly embedded in $B_{n+k}$.
However, when talking about a closed braid $\wh\beta$, the number of strands in $\beta$ does matter.
We shall use the notation $[\beta]_n$ to emphasize that $\beta$ is regarded as an $n$-braid,
and use $[\beta]_n^{\;\wh{}}$ for its closure.
For example, $[\beta]_{n+1}^{\;\wh{}}$ adds a free circle to $[\beta]_n^{\;\wh{}}$.
The Markov move says $[\beta\sigma_n^{\pm1}]_{n+1}^{\;\wh{}}$ is isotopic to $[\beta]_n^{\;\wh{}}$.

For a braid $\beta\in B_n$ and an integer $k\ge0$, we shall use
$\beta^{\triangleright k} \in B_{n+k}$ to denote the $k$-th shifted version of $\beta$, i.e.,
the braid obtained from the word $\beta$ by replacing each generator $\sigma_i$ with $\sigma_{i+k}$.
Its closure $[\beta^{\triangleright k}]_{n+k}^{\;\wh{}}$ is isotopic to
$[\beta]_{n+k}^{\;\wh{}}$.

Suppose $\beta,\beta'\in B_n$ and $\gamma\in B_p$.
Observe from the diagram defining braid closure that
the closed braid $[\beta\sigma_n^{\pm1} \gamma^{\triangleright n}\beta']_{n+p}^{\;\wh{}}$
is isotopic to $[\beta\gamma^{\triangleright (n-1)}\beta']_{n+p-1}^{\;\wh{}}$
(which is in fact a connected sum of oriented links $[\beta\beta']_n^{\;\wh{}}$ and $[\gamma]_p^{\;\wh{}}$).
By an abuse of language, we will call this a \emph{Markov move}.
If $\beta'$ brings the $n$-th position at its top to the same position at its bottom,
then $[\beta\gamma^{\triangleright (n-1)}\beta']_{n+p-1}^{\;\wh{}}$
is isotopic to $[\beta\beta'\gamma^{\triangleright (n-1)}]_{n+p-1}^{\;\wh{}}$.
We will refer to it as a \emph{slide} move
(in the connected sum,
sliding $[\gamma]_p^{\;\wh{}}$ down the last strand of $\beta'$).

\begin{lem}
\label{lem:stabilization3}
Assume that $P_L \in \mathbb Z[a^{\pm1},z^{\pm1}]$ is
an invariant of oriented links that satisfies skein relations \textup{(II)} and \textup{(III)}.
Then for $\beta\in B_n$ and $\gamma\in B_p$ we have
\[
(1+z^2-a^2)\cdot P\left( [\beta\gamma^{\triangleright n}]_{n+p}^{\;\wh{}} \right) =
(a^{-2}-1)\cdot P\left( [\beta\sigma_{n}^2\gamma^{\triangleright n}]_{n+p}^{\;\wh{}} \right) .
\]
\end{lem}

\begin{proof}
The braid form of axioms (II) and (III) are the relators
(II$_{\text{B}}$) and (III$_{\text{B}}$), respectively.
Multiplying (III$_{\text{B}}$) by $\sigma_2\sigma_1^{-1}$ on the right we get another relator
\[
a^{-1}\cdot \sigma_2^{-2}\sigma_1^2\sigma_2
+ a\cdot \sigma_2\sigma_1^{-2}
- a^{-1}\cdot \sigma_2
- a\cdot \sigma_1^2\sigma_2^{-1}\sigma_1^{-2} .
\]
It gives us an equality between the $P$'s of closed $(n+p+1)$-braids:
\begin{align*}
&a^{-1}\cdot P\left( [\beta(\sigma_{n+1}^{-2}\sigma_{n}^2\sigma_{n+1})\gamma^{\triangleright (n+1)}]_{n+p+1}^{\;\wh{}} \right)
+ a\cdot P\left( [\beta(\sigma_{n+1}\sigma_{n}^{-2})\gamma^{\triangleright (n+1)}]_{n+p+1}^{\;\wh{}} \right)
\\
& - a^{-1}\cdot P\left( [\beta\sigma_{n+1}\gamma^{\triangleright (n+1)}]_{n+p+1}^{\;\wh{}} \right)
- a\cdot P\left( [\beta(\sigma_{n}^2\sigma_{n+1}^{-1}\sigma_{n}^{-2})\gamma^{\triangleright (n+1)}]_{n+p+1}^{\;\wh{}} \right)
=0 .
\end{align*}
These closed braids can be simplified via isotopy moves
(c=conjugacy, M=Markov and s=slide):
\begin{align*}
[\beta\sigma_{n+1}^{-2}\sigma_{n}^2\sigma_{n+1}\gamma^{\triangleright (n+1)}]_{n+p+1}^{\;\wh{}}
& \overset{\text{c}}\rightsquigarrow
[\beta\sigma_{n}^2\sigma_{n+1}\gamma^{\triangleright (n+1)}\sigma_{n+1}^{-2}]_{n+p+1}^{\;\wh{}}
\\
& \overset{\text{s}}{\rightsquigarrow}
[\beta\sigma_{n}^2\sigma_{n+1}^{-1}\gamma^{\triangleright (n+1)}]_{n+p+1}^{\;\wh{}}
\overset{\text{M}}{\rightsquigarrow}
[\beta\sigma_{n}^2\gamma^{\triangleright n}]_{n+p}^{\;\wh{}} ;
\\
[\beta\sigma_{n+1}\sigma_{n}^{-2}\gamma^{\triangleright (n+1)}]_{n+p+1}^{\;\wh{}}
& \overset{\text{M}}{\rightsquigarrow}
[\beta\gamma^{\triangleright n}\sigma_{n}^{-2}]_{n+p}^{\;\wh{}}
\overset{\text{s}}\rightsquigarrow
[\beta\sigma_{n}^{-2}\gamma^{\triangleright n}]_{n+p}^{\;\wh{}} ;
\\
[\beta\sigma_{n+1}\gamma^{\triangleright (n+1)}]_{n+p+1}^{\;\wh{}}
& \overset{\text{M}}{\rightsquigarrow}
[\beta\gamma^{\triangleright n}]_{n+p}^{\;\wh{}} ;
\\
[\beta\sigma_{n}^2\sigma_{n+1}^{-1}\sigma_{n}^{-2}\gamma^{\triangleright (n+1)}]_{n+p+1}^{\;\wh{}}
& \overset{\text{M}}{\rightsquigarrow}
[\beta\sigma_{n}^2\gamma^{\triangleright n}\sigma_{n}^{-2}]_{n+p}^{\;\wh{}}
\overset{\text{s}}{\rightsquigarrow}
[\beta\gamma^{\triangleright n}]_{n+p}^{\;\wh{}} .
\end{align*}
Since $P_L$ is isotopy invariant, the above equality becomes
\begin{gather*}
a^{-1}\cdot P\left( [\beta\sigma_{n}^2\gamma^{\triangleright n}]_{n+p}^{\;\wh{}} \right)
+ a\cdot P\left( [\beta\sigma_{n}^{-2}\gamma^{\triangleright n}]_{n+p}^{\;\wh{}} \right)
= (a^{-1}+a)\cdot P\left( [\beta\gamma^{\triangleright n}]_{n+p}]_{n+1}^{\;\wh{}} \right) .
\\
\intertext{Comparing it with the equality (from (II$_{\text{B}}$))}
a^{-2}\cdot P\left( [\beta\sigma_{n}^2\gamma^{\triangleright n}]_{n+p}^{\;\wh{}} \right)
+ a^2\cdot P\left( [\beta\sigma_{n}^{-2}\gamma^{\triangleright n}]_{n+p}^{\;\wh{}} \right)
= (2+z^2)\cdot P\left( [\beta\gamma^{\triangleright n}]_{n+p}^{\;\wh{}} \right) ,
\end{gather*}
we get the desired conclusion.
\end{proof}

\begin{cor}
\label{cor:stabilization}
Under the assumption of the above lemma, the following two relations are equivalent to each other:
\begin{gather*}
P
\begin{pmatrix}
\includegraphics[height=1.2cm]{fig_012/IO-0}
\end{pmatrix}
=z^{-1}(a^{-1}-a)\cdot
P
\begin{pmatrix}
\;\;
\includegraphics[height=1.2cm]{fig_012/I-0}
\;\;
\end{pmatrix} ;
\tag*{\rm(IO)}
\\
P
\begin{pmatrix}
\includegraphics[height=1.2cm]{fig_012/Phi+-0}
\end{pmatrix}
=az^{-1}(1+z^2-a^2)\cdot
P
\begin{pmatrix}
\;\;
\includegraphics[height=1.2cm]{fig_012/I-0}
\;\;
\end{pmatrix} .
\tag*{($\Phi$)}
\end{gather*}
\end{cor}

\begin{proof}
The braid form of these two relations are, respectively,
\begin{alignat}{2}
P\left( [\beta]_{n+1}^{\;\wh{}} \right) &= z^{-1}(a^{-1}-a)\cdot P\left( [\beta]_{n}^{\;\wh{}} \right)
&&\quad \text{for any braid } \beta\in B_n;
\tag*{(IO$_{\text{B}}$)}
\\
P\left( [\beta\sigma_{n}^2]_{n+1}^{\;\wh{}} \right) &= az^{-1}(1+z^2-a^2)\cdot P\left( [\beta]_{n}^{\;\wh{}} \right)
&&\quad \text{for any braid } \beta\in B_n.
\tag*{($\Phi_{\text{B}}$)}
\end{alignat}
They are equivalent to each other by the above lemma with $[\gamma]_p:=[e]_1$.
\end{proof}

There is a parallel statement for Jones polynomial:

\begin{cor}
\label{cor:stabilization_Jones}
Assume that $V_L \in \mathbb Z[t^{\pm\frac12}]$ is
an invariant of oriented links that satisfies skein relations \textup{(II$_V$)} and \textup{(III$_V$)}.
Then the following two relations are equivalent to each other:
\begin{gather*}
V
\begin{pmatrix}
\includegraphics[height=1.2cm]{fig_012/IO-0}
\end{pmatrix}
=-(t^{\frac12}+t^{-\frac12})\cdot
V
\begin{pmatrix}
\;\;
\includegraphics[height=1.2cm]{fig_012/I-0}
\;\;
\end{pmatrix} ;
\tag*{\rm(IO$_V$)}
\\
V
\begin{pmatrix}
\includegraphics[height=1.2cm]{fig_012/Phi+-0}
\end{pmatrix}
=-t^{\frac32}(t+t^{-1})\cdot
V
\begin{pmatrix}
\;\;
\includegraphics[height=1.2cm]{fig_012/I-0}
\;\;
\end{pmatrix} .
\tag*{($\Phi_V$)}
\end{gather*}
\end{cor}

For the Alexander-Conway polynomial, we have:

\begin{cor}
\label{cor:stabilization_Alexander-Conway}
Assume that $\Delta_L \in \mathbb Z[t^{\pm\frac12}]$ is
an invariant of oriented links that satisfies skein relations \textup{(II$_\Delta$)} and \textup{(III$_\Delta$)}.
Then $\Delta(L)=0$ for any split link $L$.
In particular, the following relation holds true:
\[
\Delta
\begin{pmatrix}
\includegraphics[height=1.2cm]{fig_012/IO-0}
\end{pmatrix}
=0 .
\tag*{\rm(IO$_\Delta$)}
\]
\end{cor}

\begin{proof}
For links $L_1=[\beta]_n^{\;\wh{}}$ and $L_2=[\gamma]_p^{\;\wh{}}$,
the split link $L=L_1\sqcup L_2=[\beta\gamma^{\triangleright n}]_{n+p}^{\;\wh{}}$.
Then apply Lemma~\ref{lem:stabilization3}
with substitutions $a\mapsto 1$ and $z\mapsto (t^{\frac12}-t^{-\frac12})$.
\end{proof}

\section{Proof of Theorems \ref{thm:HOMFLY}--\ref{thm:Alexander-Conway}}
\label{sec:proof_Theorems}

We shall focus on Theorem~\ref{thm:HOMFLY}, then remark on the other two.

\begin{proof}[Proof of Theorem~\ref{thm:HOMFLY}]
Let us forget about the original definition of the skein polynomial, and
regard the symbol $P_L$ as a well-defined invariant of oriented links
which satisfies the axioms (II), (III), (IO) and (O).
By Corollary~\ref{cor:stabilization}, $P_L$ also satisfies axiom ($\Phi$).
We shall show that such an invariant $P_L$ is computable,
hence uniquely determined.

It suffices to prove the following claim by induction on $n$.
\subsection* {Inductive Claim($n$)}
For every $n$-braid $\beta\in B_n$,
$P \left( [\beta]_n^{\;\wh{}} \right)$ is computable.
\vspace{1ex}

When $n=1$, Claim($1$) is true because there is only one $1$-braid $[e]_1$.
Its closure is the trivial knot, whose $P$ must be $1$ by axiom (O).

Now assume inductively that Claim($n-1$) is true, we shall prove that Claim($n$) is also true.

Suppose $\beta$ is an $n$-braid.
By Lemma~\ref{lem:reduction}, the braid $\beta\in B_n$ is equivalent to (in a computable way)
a $\Lambda$-linear combination of
braids of the form $\alpha\sigma_{n-1}^k \gamma$ with $\alpha,\gamma\in B_{n-1}$ and $k\in\{0,\pm1,2\}$.
By Example~\ref{exam:relator_vs_relation} the (mod $\mathfrak I_n$) equivalence preserves
the $P$ of closure of braids.
So $P \left( [\beta]_n^{\;\wh{}} \right)$ is a $\Lambda$-linear combination (with computable coefficients) of
$P \left( [\alpha\sigma_{n-1}^k \gamma]_n^{\;\wh{}} \right)$'s.
For $k\in\{\pm1,0,2\}$, respectively, we have
\begin{alignat*}{2}
P \left( [\alpha\sigma_{n-1}^{\pm1} \gamma]_n^{\;\wh{}} \right)
&= P \left( [\alpha \gamma]_{n-1}^{\;\wh{}} \right)
&&\qquad\text{by isotopy},
\\
P \left( [\alpha\sigma_{n-1}^0 \gamma]_n^{\;\wh{}} \right)
&= z^{-1}(a^{-1}-a) \cdot P \left( [\alpha \gamma]_{n-1}^{\;\wh{}} \right)
&&\qquad\text{by (IO)},
\\
P \left( [\alpha\sigma_{n-1}^2 \gamma]_n^{\;\wh{}} \right)
&= az^{-1}(1+z^2-a^2) \cdot P \left( [\alpha \gamma]_{n-1}^{\;\wh{}} \right)
&&\qquad\text{by $(\Phi)$}.
\end{alignat*}

Since $P \left( [\alpha \gamma]_{n-1}^{\;\wh{}} \right)$
is computable by the inductive hypothesis Claim($n-1$),
we see $P \left( [\alpha\sigma_{n-1}^k \gamma]_n^{\;\wh{}} \right)$ is also computable.
Thus Claim($n$) is proved.

The induction on $n$ is now complete.
Hence $P$ is computable for every closed braid.
\end{proof}

\begin{rem}
The induction above, together with the reduction argument of Section \ref{sec:key_lem},
provides a recursive algorithm for computing $P \left( [\beta]_n^{\;\wh{}} \right)$.
\end{rem}

\begin{rem}
A remarkable feature of this algorithm is that
it never increases the number of components of links.
In fact, all the reductions in Section~\ref{sec:key_lem} are by
axioms (II) and (III) which respect the components,
while in this Section, components could get removed
but never added, by axioms (IO) and ($\Phi$).
So if we start off with a knot, we shall always get knots along the way,
the axioms (IO) and ($\Phi$) becoming irrelevant.
This observation works even for links with any given number of components,
once we set up a suitable normalization.
Hence the Theorem~\ref{thm:HOMFLY_knot}.
\end{rem}

\begin{rem}
For the Jones polynomial, the proof above works well with the substitutions
$a\mapsto t$ and $z\mapsto (t^{\frac12}-t^{-\frac12})$.
\end{rem}

\begin{rem}
The case of Alexander-Conway polynomial is only slightly different.
Corollary~\ref{cor:stabilization_Alexander-Conway} says (IO$_\Delta$) is a
consequence of axioms (II$_\Delta$) and (III$_\Delta$),
and ($\Phi_\Delta$) is taken as an axiom.
So the proof above also works through with the substitutions
$a\mapsto 1$ and $z\mapsto (t^{\frac12}-t^{-\frac12})$.

Actually, the argument in Section~\ref{sec:stabilizations} can be adapted to
work for Conway potential function of colored links,
to the effect that in \cite[Main Theorem]{J1},
the relation (IO) is a consequence of axioms (II) and (III)
hence can be removed from the list of axioms.
\end{rem}

\section{Proof of Theorems \ref{thm:HOMFLY_knot}--\ref{thm:Alexander-Conway_knot}}
\label{sec:proof_Theorems2}

Suppose $\mu$ is a given positive integer.
Regard $P_L$ as a well-defined invariant of oriented $\mu$-compo\-nent links
which satisfies the axioms (II), (III) and (U).
We shall temporarily expand the ring $\Lambda:=\mathbb Z[a^{\pm1},z^{\pm1}]$,
where the invariant $P_L$ takes value,
to $\wt\Lambda:=\mathbb Z[a^{\pm1},z^{\pm1},(a^{-2}-1)^{-1}]$,
to allow fractions with denominator a power of $(a^{-2}-1)$.
We shall show that such an invariant $P_L$ is computable,
hence uniquely determined.
It is the normalization (U) that brings the value $P_L$ back into the original $\Lambda$.

\begin{lem}
\label{lem:reduction_HOMFLY}
Suppose $\beta\in B_n$, $p\ge0$, and $[\beta]_{n+p}^{\;\wh{}}$ has $\mu$ components.
If $n>1$, then $P \left( [\beta]_{n+p}^{\;\wh{}} \right)$ is
computable as a $\wt\Lambda$-linear combination of terms
of the form $P \left( [\beta']_{n'+p'}^{\;\wh{}} \right)$, each with $\mu$ components,
$\beta'\in B_{n'}$, $n'<n$, $p'\ge p$, and
the $(a^{-2}-1)^{-1}$-exponent of the corresponding coefficient is at most $p'-p$.
\end{lem}

\begin{proof}
By Lemma~\ref{lem:reduction}, the braid $\beta\in B_n$ is equivalent to (in a computable way)
a $\Lambda$-linear combination of
braids of the form $\alpha\sigma_{n-1}^k \gamma$ with $\alpha,\gamma\in B_{n-1}$ and $k\in\{0,\pm1,2\}$.
So $P \left( [\beta]_{n+p}^{\;\wh{}} \right)$ is a $\wt\Lambda$-linear combination (with computable coefficients) of
$P \left( [\alpha\sigma_{n-1}^k \gamma]_{n+p}^{\;\wh{}} \right)$'s.
For $k\in\{\pm1,0,2\}$, respectively, we have
\begin{alignat*}{2}
P \left( [\alpha\sigma_{n-1}^{\pm1} \gamma]_{n+p}^{\;\wh{}} \right)
&= P \left( [\gamma\alpha \sigma_{n-1}^{\pm1}]_{n+p}^{\;\wh{}} \right)
&&\quad\text{by braid conjugation},
\\
&= P \left( [\gamma\alpha ]_{(n-1)+p}^{\;\wh{}} \right)
&&\quad\text{by Markov move},
\\
P \left( [\alpha\sigma_{n-1}^0 \gamma]_{n+p}^{\;\wh{}} \right)
&= P \left( [\alpha \gamma]_{(n-1)+(p+1)}^{\;\wh{}} \right)
&&\quad\text{obvious},
\\
P \left( [\alpha\sigma_{n-1}^2 \gamma]_{n+p}^{\;\wh{}} \right)
&= P \left( [\gamma\alpha\sigma_{n-1}^2]_{n+p}^{\;\wh{}} \right)
&&\quad\text{by braid conjugation}
\\
&= \frac{1+z^2-a^2}{a^{-2}-1} \cdot P \left( [\gamma \alpha]_{(n-1)+(p+1)}^{\;\wh{}} \right)
&&\quad\text{by Lemma~\ref{lem:stabilization3}}.
\end{alignat*}
The $P$'s on the right hand sides satisfy the required conditions.
\end{proof}

\begin{proof}[Proof of Theorem~\ref{thm:HOMFLY_knot}]
Suppose a link $L$ with $\mu$ components is presented as $[\beta]_{n}^{\;\wh{}}=[\beta]_{n+0}^{\;\wh{}}$.
Apply Lemma~\ref{lem:reduction_HOMFLY} repeatedly until no such reduction is possible.
Then $P \left( [\beta]_{n+0}^{\;\wh{}} \right)$ is computed as a $\wt\Lambda$-linear combination of
terms $P \left( [\beta']_{n'+p'}^{\;\wh{}} \right)$, each with $n'=1$ hence $\beta'=[e]_1$, such that
\begin{itemize}
\item[(1)] every $[\beta']_{n'+p'}^{\;\wh{}}=[e]_{1+p'}^{\;\wh{}}$ has $\mu$ components, hence $p'=\mu-1$,
and $[\beta']_{n'+p'}^{\;\wh{}} =[e]_\mu^{\;\wh{}} =U_\mu$;
and
\item[(2)] the $(a^{-2}-1)^{-1}$-exponent of every coefficient is at most $p'-0=\mu-1$.
\end{itemize}
Therefore, $P (L)$ is computable and, by axiom (U),
every term $P \left( [\beta']_{n'+p'}^{\;\wh{}} \right) $ has a factor $(a^{-2}-1)^{\mu-1}$
that can cancel the $(a^{-2}-1)^{-1}$-exponent in its coefficient,
so $P (L) \in \Lambda$.
\end{proof}

Theorem~\ref{thm:Jones_knot} can be proved similarly, but
Theorem~\ref{thm:Alexander-Conway_knot} needs modifications.
Define $\delta_p:=\sigma_1^2\sigma_2^2\dots\sigma_{p-1}^2 \in B_p$ whose closure
$[\delta_p]_p^{\;\wh{}}$ is the oriented $p$-component chain $C_p$.

\begin{lem}
\label{lem:reduction_Alexander-Conway}
Suppose $\beta\in B_{n+1}$, $p\ge1$, and $[\beta\delta_p^{\triangleright n}]_{n+p}^{\;\wh{}}$
has $\mu$ components.
If $n>0$, then $\Delta \left( [\beta\delta_p^{\triangleright n}]_{n+p}^{\;\wh{}} \right)$ is
computable as a $\mathbb Z[t^{\pm\frac12}]$-linear combination of terms
of the form $\Delta \left( [\beta'\delta_{p'}^{\triangleright n'}]_{n'+p'}^{\;\wh{}} \right)$,
each with $\mu$ components, $\beta'\in B_{n'+1}$, $n'<n$ and $p'\ge p$.
\end{lem}

\begin{proof}
By Lemma~\ref{lem:reduction} (with substitutions $a\mapsto 1$ and $z\mapsto (t^{\frac12}-t^{-\frac12})$),
the braid $\beta\in B_{n+1}$ is equivalent to (in a computable way)
a $\mathbb Z[t^{\pm\frac12}]$-linear combination of
braids of the form $\alpha\sigma_{n}^k \gamma$ with $\alpha,\gamma\in B_{n}$ and $k\in\{0,\pm1,2\}$.
Multiplication by $\delta_p^{\triangleright n}$ makes the braid
$\beta\delta_p^{\triangleright n}$ equivalent to a linear combination of braids
of the form $\alpha\sigma_{n}^k \gamma\delta_p^{\triangleright n}$,
so $\Delta \left( [\beta\delta_p^{\triangleright n}]_{n+p}^{\;\wh{}} \right)$ is
computed as a linear combination of the
$\Delta \left( [\alpha\sigma_{n}^k \gamma\delta_p^{\triangleright n}]_{n+p}^{\;\wh{}} \right)$'s.
For $k\in\{\pm1,0,2\}$, respectively, we have
\begin{alignat*}{2}
\Delta \left( [\alpha\sigma_{n}^{\pm1} \gamma\delta_p^{\triangleright n}]_{n+p}^{\;\wh{}} \right)
&= \Delta \left( [\gamma\alpha\sigma_{n}^{\pm1} \delta_p^{\triangleright n}]_{n+p}^{\;\wh{}} \right)
&&\quad\text{by braid conjugacy},
\\
&= \Delta \left( [\gamma\alpha \delta_p^{\triangleright (n-1)}]_{n+p-1}^{\;\wh{}} \right)
&&\quad\text{by a Markov move},
\\
\Delta \left( [\alpha\sigma_{n}^0 \gamma\delta_p^{\triangleright n}]_{n+p}^{\;\wh{}} \right)
&= 0
&&\quad\text{by Corollary~\ref{cor:stabilization_Alexander-Conway}},
\\
\Delta \left( [\alpha\sigma_{n}^2 \gamma\delta_p^{\triangleright n}]_{n+p}^{\;\wh{}} \right)
&= \Delta \left( [\gamma\alpha\sigma_{n}^2 \delta_p^{\triangleright n}]_{n+p}^{\;\wh{}} \right)
&&\quad\text{by braid conjugacy}
\\
&= \Delta \left( [\gamma\alpha \delta_{p+1}^{\triangleright (n-1)}]_{n+p}^{\;\wh{}} \right)
&&\quad\text{by definition}.
\end{alignat*}
The $\Delta$'s on the right hand sides are in the desired form with $n'=n-1$
and with $\mu$ components.
\end{proof}

\begin{proof}[Proof of Theorem~\ref{thm:Alexander-Conway_knot}]
Suppose a link $L$ with $\mu$ components is presented as
$[\beta]_{n+1}^{\;\wh{}}=[\beta\delta_1^{\triangleright n}]_{n+1}^{\;\wh{}}$.
Apply Lemma~\ref{lem:reduction_Alexander-Conway} repeatedly until no such reduction is possible.
Then $\Delta \left( [\beta\delta_1^{\triangleright n}]_{n+1}^{\;\wh{}} \right)$
is computed as a $\mathbb Z[t^{\pm\frac12}]$-linear combination of
terms $\Delta \left( [\beta'\delta_{p'}^{\triangleright n'}]_{n'+p'}^{\;\wh{}} \right)$,
with $n'=0$.
Hence each $\beta'=[e]_1$, $p'=\mu$ the number of components,
and each $[\beta'\delta_{p'}^{\triangleright n'}]_{n'+p'}^{\;\wh{}}
= [\delta_\mu]_\mu^{\;\wh{}} = C_\mu$.
Therefore $\Delta (L)$ is computable and moreover, by axiom (C$_\Delta$),
divisible by $\Delta(C_\mu) = (t^{\frac12}-t^{-\frac12})^{\mu-1}$.
\end{proof}

\end{document}